\documentclass[10 pt,a4paper,twoside,reqno]{amsart}
\usepackage{amsfonts,amssymb,amscd,amsmath,enumerate,verbatim,calc}

\usepackage{bbm}
\usepackage{graphicx}
\usepackage{varwidth}
\usepackage{hyperref}
\hypersetup{
    colorlinks=true, 
    linktoc=all,     
    linkcolor=blue,
        citecolor=red,
    filecolor=black,
    urlcolor=blue	
}

\usepackage{fancyhdr}
\usepackage[inline]{enumitem}
\makeatletter
\newcommand{\inlineitem}[1][]{%
\ifnum\enit@type=\tw@
    {\descriptionlabel{#1}}
  \hspace{\labelsep}%
\else
  \ifnum\enit@type=\z@
       \refstepcounter{\@listctr}\fi
    \quad\@itemlabel\hspace{\labelsep}%
\fi} \makeatother
\newcommand{\ga}{\alpha}

\newcommand{\gth}{\theta}

\newcommand{\gl}{\lambda}

\newcommand{\gp}{\pi}

\newcommand{\gs}{\sigma}

\newcommand{\gt}{\tau}

\newcommand{\gf}{\phi}

\newcommand{\subs}{\subset}

\newcommand{\mbb}{\mathbb}

\newcommand{\mcl}{\mathcal}

\newcommand{\us}{\underset}
\newcommand{\os}{\overset}

\newcommand{\lra}{\longrightarrow}

\newcommand{\Ra}{\Rightarrow}

\newcommand{\equ}[1]{%
\begin{equation*}
#1
\end{equation*}
}
\newcommand{\equa}[1]{%
\begin{equation*}
\begin{aligned}
#1
\end{aligned}
\end{equation*}
}

\newtheorem{theorem}{Theorem}[section]

\newtheorem{lemma}[theorem]{Lemma}

\newtheorem{example}[theorem]{Example}
\newtheorem{note}[theorem]{Note}

\theoremstyle{definition}
\newtheorem{defn}[theorem]{Definition}
\theoremstyle{remark}
\newtheorem{remark}[theorem]{Remark}
\numberwithin{equation}{section}
\makeatletter
\def\namedlabel#1#2{\begingroup
   \def\@currentlabel{#2}%
   \label{#1}\endgroup
}
\makeatother

\begin{document}
\thispagestyle{empty}

\allowdisplaybreaks

\title[A Representation Theorem]
{A Representation Theorem for Generic Line Arrangements with Global Cyclicity in the Plane}
\author[C.P. Anil Kumar]{Author: C.P. Anil Kumar*}
\thanks{*The author is supported by a research grant and facilities provided by Center for study of Science, Technology and Policy (CSTEP), Bengaluru, INDIA for this research work.}
\subjclass[2010]{Primary: 51A20, Secondary: 52C35}
\keywords{Ordered Fields, Line Arrangements in the Plane, Combinatorial Cycle Invariants, Global Cyclicity}
\begin{abstract}
In this article, we prove a representation theorem that any generic line arrangement
in the plane over an ordered field which has global cyclicity can be represented isomorphically by a line arrangement with a given set of distinct slopes of the same cardinality.
\end{abstract}
\maketitle

\section{\bf{Introduction}}
The line arrangements (refer to Definition~\ref{defn:linearrangement}) have been studied by authors like B.~Gr{\"u}nbaum~\cite{MR0307027},~\cite{MR1976856}, J.~E.~Goodman and R.~Pollack~\cite{MR0583961}, R.~P.~Stanley~\cite{MR2383131} in various contexts over fields $\mbb{Q},\mbb{R},\mbb{C}$ and finite fields $\mbb{F}_q,q$ a prime power.
This topic has applications in areas such as Combinatorics, Braids and Configurations Spaces, Computer Science and Physics. In the context of arrangements, combinatorial abstractions are studied for vector configurations and hyperplane arrangements.
Here in this article we consider generic line arrangements over ordered fields (refer to Definition~\ref{defn:OrderedField}) and prove a representation theorem by associating combinatorial invariants. This result seems to be new. In this context, the method of
associating cycle invariants as a combinatorial model to point arrangements in
the plane has already been explored by authors J.~E.~Goodman and R.~Pollack~\cite{MR0583961}. A similar method is explained in chapter $10$ of the book~\cite{MR3823190}. In the proof of the main result, here, we associate cycle invariants as a combinatorial model to generic line arrangements in the plane which have global cyclicity (refer to Definition~\ref{defn:GlobalCyclicity}).

\subsection{Definitions and the main result}
Now we mention a few definitions before we get to the statement of the main theorem.
\begin{defn}[Definition of Ordered Field]
\label{defn:OrderedField}
~\\
Let $(\mbb{F},\leq)$ be a totally ordered field. Consider the following two properties.
\begin{enumerate}
\item P1: If $x,y,z\in \mbb{F}$ then $x\leq y\Ra x+z\leq y+z$.
\item P2: If $x,y\in \mbb{F}$ then $x\geq 0,y\geq 0\Ra xy \geq 0$.
\end{enumerate}
In this article, an ordered field $\mbb{F}$ means a totally ordered field $(\mbb{F},\leq)$ which satisfies the two properties P1,P2.
For example any subfield of $\mbb{R}$ is an ordered field with the induced ordering from the field of reals.
\end{defn}
\begin{remark}
Chapter XI in S.~Lang~\cite{MR1878556}, Chapters $5,11$ in N.~Jacobson~\cite{MR0780184},~\cite{MR1009787} respectively gives more interesting properties of such ordered fields.
\end{remark}

\begin{defn}[Lines in Generic Position in the Plane $\mbb{F}^2$ or Generic Line Arrangement]
\label{defn:linearrangement}
Let $\mbb{F}$ be an ordered field. Let $n$ be a positive integer.
We say a finite set $\mcl{L}^{\mbb{F}}_n=\{L_1,L_2,\ldots,L_n\}$ of lines in $\mbb{F}^2$ is in a generic position
or is a line arrangement if the following two conditions hold.
\begin{enumerate}
\item No two lines are parallel.
\item No three lines are concurrent.
\end{enumerate}
In this case we say that $\mcl{L}^{\mbb{F}}_n$ is a line arrangement. We denote the line arrangement
by $\mcl{L}_n$ if the field $\mbb{F}=\mbb{R}$. We say $n$ is the cardinality of the line arrangement.
\end{defn}
\begin{remark}
Henceforth in this article, a line arrangement always means a generic line arrangement.
\end{remark}

Here we give the definition of an isomorphism between two line arrangements.
\begin{defn}
\label{defn:Iso}
~\\
Let $\mbb{F}$ be an ordered field. Let $n,m$ be positive integers. Let \equ{(\mcl{L}_n^{\mbb{F}})_1=\{L_1,L_2,\ldots,L_n\},(\mcl{L}_m^{\mbb{F}})_2=\{M_1,M_2,\ldots,M_m\}}
be two line arrangements in the plane $\mbb{F}^2$ of cardinalities $n,m$ respectively.
We say a map $\gf:(\mcl{L}_n^{\mbb{F}})_1 \lra (\mcl{L}_m^{\mbb{F}})_2$ is an isomorphism between the line arrangements if
\begin{enumerate}
\item the map $\gf$ is a bijection, (that is, $n=m$) with $\gf(L_i)=M_{\gf(i)},1\leq i\leq n$ and
\item for any $1\leq i\leq n$ the order of intersection vertices on the lines $L_i,M_{\gf(i)}$ agree via the bijection induced by $\gf$ on its subscripts.
There are four possibilities of pairs of orders and any one pairing of orders out of the four pairs must agree via the bijection induced by $\gf$ on its subscripts.
\end{enumerate}
Two mutually opposite orders of points arise on any line in the plane because the field $\mbb{F}$ is ordered.
\end{defn}

We define a line arrangement in the plane which has global cyclicity as follows.
\begin{defn}[Existence of Global Cyclicity]
	\label{defn:GlobalCyclicity}
	~\\
	Let $\mbb{F}$ be an ordered field. Let $\mcl{L}_n^{\mbb{F}}=\{L_1,L_2,\ldots,L_n\}$ be a line arrangement. We say that there exists global cyclicity in the line arrangement $\mcl{L}_n^{\mbb{F}}$
	if all the lines form the sides of a convex polygon in some cyclic order of the lines.
\end{defn}

Main Theorem~\ref{theorem:repclass} is regarding a representation of a line
arrangement which has global cyclicity, isomorphically (refer to Definition~\ref{defn:Iso}),
by some set of lines forming a line arrangement with a given set of distinct slopes,
of same cardinality, which is useful to pick an element in the same isomorphism class by fixing
a finite set of slopes. The representation theorem is proved after the proof of
Theorem~\ref{theorem:TransParallel}. Now we state the theorem here.

\begin{theorem}[Representation Theorem]
\label{theorem:repclass}
~\\
Let $\mbb{F}$ be an ordered field. Let \equ{\{m_1,m_2,\ldots,m_n\}\subs \mbb{F}\cup \{\infty\}} be a set of $n$ distinct slopes. Then in any isomorphism class 
of a line arrangement $\mcl{L}_n^{\mbb{F}}$ of cardinality $n$ which has global cyclicity, there exists a set of $n$ lines which represents exactly this slope set.
\end{theorem}
\subsection{Some more definitions}
~\\
We introduce the next two definitions regarding any polygon made up of lines.
The usefulness of the first definition lies in identifying convex polygons in a line arrangement.
\begin{defn}[$2\operatorname{-}$standard Consecutive Structure on Slopes]
Let $\mbb{F}$ be an ordered field. We say a certain $n\operatorname{-}$tuple of slopes
\equ{(m_1,m_2,\ldots,m_n) \in (\mbb{F}\cup \{\infty\})^{n}=(\mbb{PF}^1_{\mbb{F}})^n} has a $2\operatorname{-}$standard consecutive
structure if the following occurs.
\begin{enumerate}
\item $0\leq m_1<m_2<\ldots <m_i \leq \infty$
\item $m_{i+1}<m_{i+2}<\ldots < m_j \leq 0$
\item $0<m_{j+1}<m_{j+2}<\ldots < m_k \leq \infty$
\item $m_{k+1}<m_{k+2}<\ldots < m_n<0$
\end{enumerate}
\end{defn}
for some $1\leq i<j<k\leq n$ if they exist. The second sequence of slopes may be empty, that is, $j+1=i+1$ with $0<m_{j+1}$. The third sequence of slopes may be empty with $j+1=k+1$ and $m_{k+1}<m_{j}\leq 0$.
Or the last sequence of slopes may be empty, that is, $k=n$.
If the slopes $m_i:1\leq i \leq n$ arise from a line arrangement then the slopes are distinct as any two
distinct lines meet. The structure is considered as $2\operatorname{-}$standard by referring to usual angles
instead of slopes. Refer to Figure~\ref{fig:One}.
\begin{remark}
If the slopes $m_i$ of the lines $L_i$ arise from a line arrangement $\mcl{L}_n^{\mbb{F}}=\{L_1,L_2,\ldots,L_n\}$ with global cyclicity given by the anti-clockwise order
\equ{L_1\lra L_2\lra \ldots \lra L_n\lra L_1}
with $m_1$ being the least non-negative slope then we have $m_j<0\leq m_1<m_{j+1}$ whenever the second sequence of slopes exists. Also between the last slopes in the first and third sequences,
at most one of $m_i$ and $m_k$ can be infinity as the slopes are all distinct. We remark here that we can have line arrangements where all the slopes $m_i$ are negative for $1\leq i\leq n$ in which case the first and third sequences of slopes are empty.
\end{remark}
\begin{figure}[h]
\centering
\includegraphics[width = 1.05\textwidth]{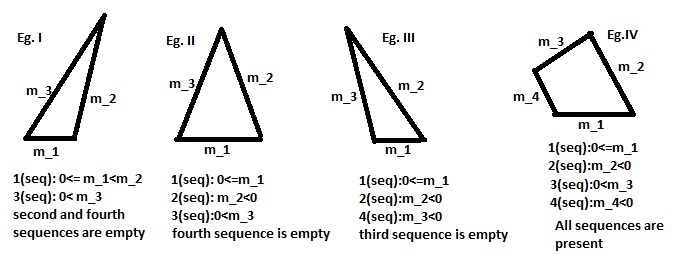}
\caption{Some Examples of $2\operatorname{-}$standard Consecutive Structure on Slopes with $m_1=0$}
\label{fig:One}
\end{figure}
Now we introduce the second definition. This definition is useful in classifying the cycles at infinity of a line arrangement
\equ{\mcl{L}_n^{\mbb{F}}=\{L_1,L_2,\ldots,L_n\}}
which has global cyclicity in this anti-clockwise cyclic order
\equ{L_1\lra L_2 \lra \ldots \lra L_n \lra L_1.}
\begin{defn}[Opposite Vertex of a Side in a Convex Polygon]
~\\
Let $\mbb{F}$ be an ordered field.
Let $\mcl{L}_n^{\mbb{F}}=\{L_1,L_2,\ldots,L_n\}$ be a line arrangement. Let
\equ{\{L_{1}\lra L_{2}\lra \ldots \lra L_{n}\lra L_{1}\}}
be a convex polygon with $n\operatorname{-}$sides in this anticlockwise
cyclic order. We assume that $L_{1}$ has the least non-negative slope $m_1$.
Suppose the $2\operatorname{-}$standard consecutive structure
for this polygon is given by
\begin{enumerate}
\item $0\leq m_1<m_2<\ldots <m_i \leq \infty$
\item $m_{i+1}<m_{i+2}<\ldots < m_j < 0$
\item $0<m_{j+1}<m_{j+2}<\ldots < m_k \leq \infty$
\item $m_{k+1}<m_{k+2}<\ldots < m_n<0$
\end{enumerate}
for some $1\leq i<j<k\leq n$, if they exist. Here $m_l$ is the slope of the line $L_{l}$. Then the opposite vertex
is the intersection of the pair of lines corresponding to the two extreme vertices on the line $L_1$ of the line arrangement.
If the second and the third sequences of the slopes are non-empty then the opposite vertex is $L_{j}\cap L_{j+1}$ and
we just note that \equ{m_j< 0\leq m_1< m_{j+1}.}
The definition of the opposite vertex for the remaining sides of the
polygon is similar using extreme vertices on the respective lines.
\end{defn}
We define a line at infinity for a line arrangement. This is useful to define the cycle at infinity for a line arrangement
which has global cyclicity.
\begin{defn}[Line at Infinity]
~\\
Let $\mbb{F}$ be an ordered field. Let $\mcl{L}_n^{\mbb{F}}=\{L_1,L_2,\ldots,L_n\}$. We say a line $L$ is a line at infinity
if $\mcl{L}_n^{\mbb{F}} \cup \{L\}$ is a line arrangement and all the vertices, that is, zero dimensional intersections of the arrangement $\mcl{L}_n^{\mbb{F}}$ lie
on one side of $L$ (possibly including $L$).
\end{defn}

\begin{note}
All lines not passing through the origin pass through two or three quadrants. They are oriented in the direction of the
quadrants accordingly as
\begin{itemize}
\item $I\lra II,II\lra III,III\lra IV, IV\lra I$ if the line passes through only two quadrants.
\item $I\lra II\lra III,II\lra III\lra IV,III\lra IV\lra I, IV \lra I \lra II$ if the line passes through the exactly three quadrants.
\end{itemize}
So the origin is always on the left side of the oriented line.
\end{note}

Define the cycle at infinity as follows.
\begin{defn}[Cycle at infinity: An Element of the Symmetric Group]
\label{defn:CAI}
~\\
Let $\mbb{F}$ be an ordered field. Let $\mcl{L}_n^{\mbb{F}}=\{L_1,L_2,\ldots, L_n\}$ be a line arrangement with global cyclicity
in this anti-clockwise cyclic order
\equ{L_1\lra L_2 \lra \ldots \lra L_n \lra L_1.}
Let $L$ be any line at infinity and not passing through origin which is oriented
and right side (the non-origin side) of the discrete half sides of $L$ is empty, that is,
does not contain any vertices of the line arrangement $\mcl{L}_n^{\mbb{F}}$. Then the cycle at
infinity is defined as the sequence of subscripts of the lines, a permutation
\equ{(i_1i_2\cdots i_n)\in S_n}
corresponding to the intersections of the lines in $\mcl{L}_n^{\mbb{F}}$ with $L$ in the direction
of the orientation of $L$.
\end{defn}

\section{\bf{On the $i\operatorname{-}$standard consecutive structure of an $n\operatorname{-}$cycle}}
Here in this section, we define $i\operatorname{-}$standard consecutive structure on an $n\operatorname{-}$cycle and prove Theorem~\ref{theorem:iStandardConsecutiveCycle} that
an $n\operatorname{-}$cycle has an unique $i\operatorname{-}$standard consecutive structure for some $1\leq i\leq n-1$ for $n>1$.

Now we introduce a structure on a permutation as follows.
\begin{defn}
We say an $n\operatorname{-}$cycle $(a_1=1,a_2,\cdots,a_n)$ is an $i\operatorname{-}$standard cycle if there exists a way to write
the integers $a_i:i=1,\cdots,n$ as $i$ sequences of inequalities as follows:
\equa{a_{11}<a_{12}&<\cdots<a_{1j_1}\\
a_{21}<a_{22}&<\cdots<a_{2j_2}\\
a_{31}<a_{32}&<\cdots<a_{3j_3}\\
&\vdots \\
a_{i1}<a_{i2}&<\cdots<a_{ij_i}}
where $\{a_{st}\mid 1\leq s\leq i,1\leq t\leq j_s\}=\{a_1,a_2,\cdots,a_n\}=\{1,2,\cdots,n\},
j_1+j_2+\cdots+j_i=n$ and $i$ is minimal, that is, there exists no smaller integer with such property
and further more that $a_{s(t+1)}$ occurs to the right of $a_{st}$ for every $1\leq s\leq i$ and
$1\leq t \leq j_s-1$ in this cycle arrangement $(a_1=1,a_2,\cdots,a_n)$.
\end{defn}
Define $i\operatorname{-}$standard consecutive structure on an $n\operatorname{-}$cycle as follows.
\begin{defn}
\label{defn:2StandardConsecutiveStructure}
We say an $n\operatorname{-}$ cycle $(a_1=1,a_2,\cdots,a_n)$ is a consecutive $i\operatorname{-}$standard cycle
or a $i\operatorname{-}$standard consecutive cycle if we have
\equ{a_{s1}<a_{s2}<\cdots<a_{sj_s}}
and in addition $a_{st}=a_{s1}+(t-1),1\leq t\leq j_s,1\leq s\leq i$
where $\{a_{st}\mid 1\leq s\leq i,1\leq t\leq j_s\}=\{a_1,a_2,\cdots,a_n\}=\{1,2,\cdots,n\},
j_1+\cdots+j_s=n$ and $a_{s(t+1)}$ occurs to the right of $a_{st}$ for every $s=1,\cdots,i$ and
$1\leq t \leq j_s-1$ in this cycle arrangement $(a_1=1,a_2,\cdots,a_n)$
and $i$ is minimal, that is, there exists no smaller integer with such property.
If the minimal value of $i$ is two then we say that the cycle has $2\operatorname{-}$standard consecutive structure.
\end{defn}
\begin{example}
For example if we consider the $5\operatorname{-}$cycle $(1,4,5,2,3)$ it is a $2\operatorname{-}$standard consecutive cycle.
However it has the following two $2\operatorname{-}$standard structures.
\begin{itemize}
\item $1<4<5,2<3$ (not consecutive).
\item $1<2<3,4<5$ (consecutive).
\end{itemize}
\end{example}
Now we prove Theorem~\ref{theorem:iStandardConsecutiveCycle} on the existence and uniqueness of the $i\operatorname{-}$standard consecutive structure on an $n\operatorname{-}$cycle for $n>1$.
\begin{theorem}[Existence and Uniqueness of the Consecutive $i\operatorname{-}$Standard Structure on an $n\operatorname{-}$cycle]
\label{theorem:iStandardConsecutiveCycle}
~\\
For $n>1$, there exists $i\operatorname{-}$standard consecutive structure for some $1\leq i \leq n-1$ on an $n\operatorname{-}$cycle and is uniquely determined.
\end{theorem}
\begin{proof}
We prove this by induction on $i,n$ as follows. If $i=n=1$ then there is nothing to prove.
The position of the element $n$ is uniquely determined as it
should appear in one of them at the end and $(n-1)$ appears before $n$ if $(n-1)$
appears before $n$ in the $n\operatorname{-}$cycle and appears as a single element of standardness
if $(n-1)$ appears after $n$. Now we remove $n$ from the cycle. The remaining cycle
is either $i\operatorname{-}$standard on $(n-1)\operatorname{-}$elements or $(i-1)$ standard on $(n-1)\operatorname{-}$elements.
This proves the theorem.

We can actually build this structure in an unique way for the
given $n\operatorname{-}$cycle as follows. Write $1$ first. Then write $1<2$ as it appears later. Then write $3$ next to $2$ if it appears
after $2$ or write as a single element of standardness if it appears before $2$ and so on.
\end{proof}
\section{\bf{On the cycle at infinity and the $2\operatorname{-}$standard consecutive structure}}
Here in the section we prove in Theorem~\ref{theorem:CIOVGG} that the cycle at infinity determines the line arrangement which contain global cyclicity up to a cyclic
renumbering of the subscripts of the lines.

Now we prove Theorem~\ref{theorem:A} about cycles at infinity arising out line arrangements with global cyclicity.
\begin{theorem}
\label{theorem:A}
Let $\mbb{F}$ be an ordered field. Let $\mcl{L}_n^{\mbb{F}}=\{L_1,L_2,\cdots,L_n\}$ be a line arrangement which gives rise to global cyclicity
in this anticlockwise manner
\equ{L_1\lra L_2 \lra \cdots \lra L_n\lra L_1}
then the cycle at infinity has a $2\operatorname{-}$standard consecutive structure.
\end{theorem}
\begin{proof}
Assume by rotation that $L_1$ has the least non-negative slope.
Suppose \equ{L_i\cap L_{i+1}} is the opposite vertex for $L_1$.
Then on the $n\operatorname{-}$ cycle at infinity we have the following unique $2\operatorname{-}$standard consecutive structure.
\begin{itemize}
\item $1<2<\cdots<i$
\item $i+1<i+2<\cdots<n$
\end{itemize}
This proves the theorem.
\end{proof}
Now we mention a note about existence of rotations.
\begin{note}
To describe rotations of the plane $\mbb{F}^2$ we need square roots of elements of $\mbb{F}$ in the field itself.
For this purpose we note that for an ordered field $\mbb{F}$, we have $char(\mbb{F})=0$ with $\mbb{Q}\subs \mbb{F}$.
Let \equ{\mcl{TAN}=\{m\in \mbb{F}^{+}\mid 1+m^2=\square\}.}
Given a slope $m>0$, if $\mbb{Q}$ is dense, then there exists an $m_1\in \mcl{TAN}\cap \mbb{Q}$ such that $0\neq m_1<\frac{m}{2}$. So arbitrary small rotations exist. Otherwise we replace $\mbb{F}$ by its real closure in which case rotations corresponding to
all slopes $m\in \mbb{F}$ exist because $\sqrt{1+m^2}\in \mbb{F}$ for all $m\in \mbb{F}$. The proofs we present here are independent of the ordered field $\mbb{F}$.
\end{note}

Now define the combinatorial relevance of the $2\operatorname{-}$standard consecutive structure with respect to geometric $2\operatorname{-}$standard consecutive structure
on slopes of lines in a line arrangement as follows.
\begin{defn}[Slope Property]
\label{defn:SlopeProperty}
~\\
Let $\mbb{F}$ be an ordered field. Let $\mcl{L}_n^{\mbb{F}}=\{L_1,L_2,\cdots,L_n\}$ be a line arrangement which gives rise to global cyclicity
in this anticlockwise manner
\equ{L_1\lra L_2 \lra \cdots \lra L_n\lra L_1.}
We say that the $2\operatorname{-}$standard consecutive structure on a permutation $n\operatorname{-}$cycle associated to a line
arrangement respects the slope property if the following occurs.
If the $2\operatorname{-}$standard consecutive structure is given by
\begin{itemize}
\item $1<2<3<\cdots<j$.
\item $j+1<j+2<\cdots<n$.
\end{itemize}
then the intersection vertices on the line $L_1$ has the following order.
\equa{&L_{j+1}\cap L_1\lra L_{j+2} \cap L_1\lra \cdots \lra L_k \cap L_1 \lra L_{k+1} \cap L_1 \lra\cdots \lra\\
&L_n\cap L_1 \lra L_2\cap L_1 \lra \cdots \lra L_i\cap L_1 \lra L_{i+1}\cap L_1 \lra \cdots \lra L_j\cap L_1,}
that is, modulo a rotation of the plane $\mbb{F}^2$, (with $m_1$, the slope of $L_1$ as the least non-negative slope), we have
\begin{enumerate}
\item $0\leq m_1<m_2<\cdots <m_i \leq \infty$
\item $m_{i+1}<m_{i+2}<\cdots < m_j < 0$
\item $0<m_{j+1}<m_{j+2}<\cdots < m_k \leq \infty$
\item $m_{k+1}<m_{k+2}<\cdots < m_n<0$
\end{enumerate}
with $m_j<0\leq m_1<m_{j+1}$ if the second and third sequences of slopes are non-empty. Here $m_i$ denotes the slope of the line $L_i$.
\end{defn}
Now we prove the theorem of this section.
\begin{theorem}[Cycle at Infinity and the Opposite Vertices of sides of the Global Cyclicity]
\label{theorem:CIOVGG}
~\\
Let \equ{\mcl{L}_n^{\mbb{F}}=\{L_1\lra L_2\lra\cdots\lra L_n\lra L_1\}}
be a line arrangement in the plane giving rise to global cyclicity in this anticlockwise cyclic order.
Then the cycle at infinity having the $2\operatorname{-}$standard structure which respects the slope property determines
uniquely the opposite vertex for any side in the $n\operatorname{-}$gon. Conversely if we know the opposite vertex for any side
in this $n\operatorname{-}$gon then the cycle at infinity is determined uniquely and its $2\operatorname{-}$standard structure
respects the slope property (refer to Definition~\emph{\ref{defn:SlopeProperty}}). Moreover the line arrangement
$\mcl{L}_n^{\mbb{F}}$ is also determined up to an isomorphism, that is, the following. Let \equ{(\mcl{L}_n^j)^{\mbb{F}}=\{L^j_1\lra L^j_2\lra\cdots\lra L^j_n\lra L^j_1\},j=1,2} be two line arrangements in the plane giving rise to global cyclicity in this anticlockwise cyclic order with
$\gs^j,j=1,2$ the cycles at infinity respectively. Let $\gf:(\mcl{L}_n^1)^{\mbb{F}}\lra (\mcl{L}_n^2)^{\mbb{F}}$ be the bijection taking $L^1_i\lra L^2_i,1\leq i\leq n$. Then we have that the following are equivalent.
\begin{enumerate}
\item $\gf$ is an isomorphism.
\item $\gs^1=\gs^2$.
\end{enumerate}
\end{theorem}
\begin{proof}
Respecting the slope property is a given, once the cycle at infinity is determined.
The cycle at infinity determines uniquely the opposite vertex
for any side in the $n\operatorname{-}$gon as we can use the $2\operatorname{-}$standard structures which respects the slope property on each
of the conjugate cycles of the cycle at infinity obtained by cyclically changing the indices $1,2,\cdots,n$.

Conversely if we know the opposite vertex for any side we need
to determine the slope ordering of the lines $L_1,L_2,\cdots,L_n$.
This slope order gets determined because of the following reason.  If we know the opposite vertex
for a side on a line $L$ of the arrangement then the sequence of intersections gets determined including the
end-points using the opposite vertex on the line $L$. Hence the complete line arrangement gets determined
using Definition~\ref{defn:Iso}. Hence this determines the cycle at infinity as well.
\end{proof}
\begin{defn}[2-standard consecutive $n\operatorname{-}$cycles]
\label{defn:2StandardConsecutiveCycles}
~\\
Let $T_n\subs S_n$ be the set of $2\operatorname{-}$standard consecutive $n\operatorname{-}$cycles in $S_n$.
\end{defn}
Now we prove a lemma.
\begin{lemma}
We have $\#(T_n)=2^{n-1}-n$.
\end{lemma}
\begin{proof}
This follows by counting the cardinality of $T_n$.
If the $2\operatorname{-}$standard consecutive structure is given by
\begin{itemize}
\item $1<2<3<\cdots<j$.
\item $j+1<j+2<\cdots<n$.
\end{itemize}
then the number of such cycles is given by $\binom{n-1}{j-1}-1$.
Hence the total number is given by \equ{\us{i=2}{\os{n-1}{\sum}}\bigg(\binom{n-1}{i-1}-1\bigg)=\us{i=0}{\os{n-1}{\sum}}\bigg(\binom{n-1}{i}-1\bigg)=2^{n-1}-n.}
\end{proof}
\section{\bf{Representation theorem}}
\label{sec:RT}
In this section we prove the main Theorem~\ref{theorem:repclass}.
First we prove the following theorem on line arrangements with global cyclicity which differ by translations.
\begin{theorem}[Transitivity on the $n\operatorname{-}$cycles which have $2\operatorname{-}$standard consecutive structures
	by translations]
	\label{theorem:TransParallel}
	~\\
	Let $\mbb{F}$ be an ordered field.
	Let \equ{\mcl{L}_n^{\mbb{F}}=\{L_1\lra L_2\lra\cdots\lra L_n\lra L_1\}}
	be a line arrangement in the plane giving rise to global cyclicity in this anticlockwise cyclic order.
	Let $\gt$ be the cycle at infinity having a $2\operatorname{-}$standard consecutive structure which respects slope property.
	Let $\gs$ be another $n\operatorname{-}$cycle having a $2\operatorname{-}$standard consecutive structure. Assume that $L_1$ has the least
	non-negative slope. Then we can move the lines
	\equ{L_2,L_3,\cdots,L_n}
	by translations into another line arrangement which after a permutation of subscripts
	\equ{2,3,\cdots,n} has global cyclicity in the anticlockwise cyclic order
	\equ{L_1\lra L_2\lra\cdots\lra L_n\lra L_1}
	and has cycle at infinity $\gs$ having the $2\operatorname{-}$standard consecutive structure which
	respects slope property.
\end{theorem}
\begin{proof}
	Consider the $n\operatorname{-}$cycles $\gt$ and $\gs$ which have the $2\operatorname{-}$standard consecutive structures.
	We claim that there exist translations of the set of
	lines of the arrangement which give global cyclicity
	\equ{L_1\lra L_2\lra \cdots \lra L_n\lra L_1}
	in the anticlockwise order after a suitable permutation of subscripts and
	having any $n\operatorname{-}$cycle with a $2\operatorname{-}$standard consecutive
	structure which respects slope property.
	
	To observe this fact first we consider over the field of reals in which
	we consider arbitrary $n\operatorname{-}$distinct angles in $[0,\gp)$
	in the increasing order corresponding to $n\operatorname{-}$lines in the real plane (say)
	\equ{0=\gth_1<\gth_2<\cdots <\gth_n<\gp.}
	It does not matter what the exact angles are, however what matters is the order of the angles
	with respect to subscripts. Now the possibilities of the $n\operatorname{-}$gons are precisely all the possibilities
	which satisfy the following.
	\begin{itemize}
		\item $0 = \ga_1<\ga_2<\cdots <\ga_i<\gp$
		\item $0<\ga_{i+1}<\ga_{i+2}<\cdots<\ga_n<\gp$.
	\end{itemize}
	where $\{\ga_1=0,\ga_2,\cdots,\ga_n\}=\{\gth_1=0,\gth_2,\cdots,\gth_n\}$.
	The lines with slopes $\ga_i:i=1,2,\cdots,n$ gives an anticlockwise $n\operatorname{-}$gon
	\equ{L_1\lra L_2\lra \cdots \lra L_n\lra L_1}
	where $L_i$ makes an angle $\ga_i$ with respect to $X\operatorname{-}$axis.
	The permutation $\gl$ of the subscripts corresponding to $\gth_i=\ga_{\gl(i)}$ are precisely those $\gl$ such that the $n\operatorname{-}$cycle $(1=\gl(1),\gl(2),\cdots,\gl(n))$ has a
	$2\operatorname{-}$standard consecutive structure. This observation can be extended to any ordered field $\mbb{F}$ where we use slopes instead
	of angles. The actual values of slopes do not matter for any $n\operatorname{-}$distinct values in $\mbb{F}\cup \{\infty\}$.
	
	So we can translate the lines to obtain the required line arrangement with required global cyclicity.
	This proves the theorem.
\end{proof}

Now we prove the Representation Theorem~\ref{theorem:repclass}.
\begin{proof}
Theorem~\ref{theorem:TransParallel} proves that all line arrangements with global cylicity can be realised by any finite set of distinct slopes of the same cardinality. This proves Theorem~\ref{theorem:repclass}.
\end{proof}

\vspace{.2cm}
\noindent C.P. ANIL KUMAR,\\
Center for Study of Science, Technology and Policy,\\
\noindent	\# 18 \& \#19, 10th Cross, Mayura Street,\\
\noindent	Papanna Layout, Nagashettyhalli, RMV II Stage,\\
\noindent	Bengaluru - 560094, INDIA.\\
E-mail: {\it akcp1728@gmail.com}
\end{document}